\documentclass[11pt,reqno]{amsart}

\usepackage[T1]{fontenc}
\usepackage{lmodern}
\usepackage{microtype}
\usepackage{amsmath,amssymb,mathtools}
\usepackage{enumitem}
\usepackage[hidelinks]{hyperref}
\usepackage[nameinlink,noabbrev]{cleveref}
\hypersetup{
  pdftitle={Sharp directed distances between degree classes of critical Sobolev maps on spheres},
  pdfauthor={Rupert L. Frank and Paata Ivanisvili},
  pdfsubject={Critical Sobolev maps, degree classes, directed distances, and bubbling},
  pdfkeywords={Sobolev maps, topological degree, homotopy classes, directed distance, bubbling, concentration}
}

\numberwithin{equation}{section}

\newtheorem{theorem}{Theorem}[section]
\newtheorem{Proposition}[theorem]{Proposition}
\newtheorem{lemma}[theorem]{Lemma}

\newcommand{\Sph}{\mathbb S}
\newcommand{\R}{\mathbb R}
\newcommand{\Z}{\mathbb Z}
\newcommand{\E}{\mathcal E}
\newcommand{\D}{\mathrm D}
\newcommand{\dd}{\,\mathrm d}
\newcommand{\distdir}{\operatorname{Dist}}
\newcommand{\weakstar}{\mathrel{\stackrel{*}{\rightharpoonup}}}
\newcommand{\Jac}{\operatorname{Jac}}

\newcommand{\so}{\mathfrak{so}}
\newcommand{\VMO}{\mathrm{VMO}}
\newcommand{\esssup}{\operatorname*{ess\,sup}}

\title[Homotopy classes of Sobolev maps on spheres]{The distance between homotopy classes\\ of Sobolev maps on spheres}

\author{Rupert L. Frank}
    \address[Rupert L. Frank]{Mathe\-matisches Institut, Ludwig-Maximilians Universit\"at M\"unchen, The\-resien\-str.~39, 80333 M\"unchen, Germany, and Munich Center for Quantum Science and Technology, Schel\-ling\-str.~4, 80799 M\"unchen, Germany, and Mathematics 253-37, Caltech, Pasa\-de\-na, CA 91125, USA}
    \email{r.frank@lmu.de}

\author[P.~Ivanisvili]{Paata Ivanisvili}
\address[P.~Ivanisvili]{Department of Mathematics, University of California, Irvine, 510C Rowland Hall, Irvine, CA 92697-3875, USA}
\email{pivanisv@uci.edu}


\begin{document}

\begin{abstract}
    We consider self-maps of a sphere in the critical Sobolev space with a given Brouwer degree. Our main result is that the (directed) distance between maps of different degrees is equal to an explicit constant times the difference in degrees. In the case of the 2-sphere this resolves an open problem by Brezis.
\end{abstract}

\maketitle

\section{Introduction and main result}\label{sec:introduction}

\subsection{The critical degree classes}
In this paper we are interested in maps from the unit sphere $\Sph^n=\{x\in\R^{n+1}:|x|=1\}$, $n\geq 2$, to itself that belong to the Sobolev space
\begin{equation}\label{eq:critical-space}
  W^{1,n}(\Sph^n;\Sph^n)
  :=\{u\in W^{1,n}(\Sph^n;\R^{n+1}):|u|=1\text{ a.e.}\}.
\end{equation}
As we will recall in detail in Section \ref{sec:preliminaries} below, such maps have a well-defined Brouwer degree. For $d\in\Z$, set
\begin{equation}\label{eq:degree-class}
  \E_d^{(n)}
  :=\{u\in W^{1,n}(\Sph^n;\Sph^n):\deg u=d\}.
\end{equation}
Our main result gives an explicit formula for the min-max quantity
$$
\sup_{f\in\E_{d_1}^{(n)}}
    \inf_{g\in\E_{d_2}^{(n)}} \int_{\Sph^n}|\D_\tau f-\D_\tau g|^n\,\dd\sigma \,,
$$
where $D_\tau u$ denotes the weak (tangential) derivative of $u\in W^{1,n}(\Sph^n;\Sph^n)$ and we employ the Hilbert--Schmidt norm, that is, if $(\tau_1,\dots,\tau_n)$ is a local orthonormal frame of $T\Sph^n$, then
\[
  |\D_\tau f - \D_\tau g|^2=\sum_{i=1}^n|\partial_{\tau_i}f - \partial_{\tau_i} g|^2.
\]
Also, $\dd\sigma$ denotes surface measure on $\Sph^n$.

Introducing the constant
$$
\kappa_n:=n^{n/2}|\Sph^n|= n^{n/2} \frac{2\pi^{(n+1)/2}}{\Gamma((n+1)/2)},
$$
we can state our main result.

\begin{theorem}\label[theorem]{thm:main}
Let $n\ge2$ and $d_1,d_2\in\Z$.  Then
\begin{equation}\label{eq:main}
\sup_{f\in\E_{d_1}^{(n)}}\inf_{g\in\E_{d_2}^{(n)}}
  \int_{\Sph^n}|\D_\tau f-\D_\tau g|^n\,\dd\sigma
  =\kappa_n \, |d_1-d_2| \,.
\end{equation}
\end{theorem}

The equality \eqref{eq:main} is also true for $n=1$, as shown in \cite[Theorem 1.4(1)]{BrezisMironescuShafrir2016}. For $n=2$, the inequality $\leq$ in \eqref{eq:main} was already known (see \cite{Brezis2023} and references therein) and the inequality $\geq$ was known when $0\leq d_1<d_2$ (see \cite[Proposition 7.3(2)]{BrezisMironescuShafrir2016}), but the inequality $\geq$ in the general case is new and gives an affirmative answer to Problem 5.9 on Brezis's list `Some of my favorite open problems' \cite{Brezis2023}. To the best of our knowledge, before the present work no exact all-degree formula was known beyond the one-dimensional case and the two-dimensional ranges above.

Theorem \ref{thm:main} has the following consequence. If we set, following \cite{BrezisMironescuShafrir2016},
\begin{equation}
    \label{eq:directed-distance}
    \distdir_n(\E_{d_1}^{(n)},\E_{d_2}^{(n)})
  :=\sup_{f\in\E_{d_1}^{(n)}}
    \inf_{g\in\E_{d_2}^{(n)}} \left( \int_{\Sph^n}|\D_\tau f-\D_\tau g|^n\,\dd\sigma \right)^{1/n},
\end{equation}
then we obtain the symmetry relation
\begin{equation}
    \label{eq:symmetric}
    \distdir_n(\E_{d_1}^{(n)},\E_{d_2}^{(n)}) = \distdir_n(\E_{d_2}^{(n)},\E_{d_1}^{(n)}) \,.
\end{equation}
Indeed, by Theorem \ref{thm:main}, both sides are equal to $(\kappa_n \, |d_1-d_2|)^{1/n}$. The equality \eqref{eq:symmetric} is new for all $n\geq 3$. In particular, it gives an affirmative answer to Open Problem 1 in \cite{BrezisMironescuShafrir2016} in the case $s=1$, $p=n$. Also, Theorem \ref{thm:main} answers Open Problem 2 in \cite{BrezisMironescuShafrir2016} for $p=n$.


\subsection*{Background}
Before describing the strategy of our proof, let us mention some previous results on this and related problems. The interaction between degree and concentration in critical Sobolev spaces arose prominently in the work of Brezis and Coron on harmonic maps \cite[Lemmas~1--2]{BrezisCoron1983}; related relaxation phenomena were developed by Bethuel, Brezis, and Coron \cite[Theorems~2--3]{BethuelBrezisCoron1990}.  Brezis and Nirenberg subsequently developed degree theory in $\VMO$ \cite[Section~I.3 and Theorem~1]{BrezisNirenberg1995}.

Brezis, Mironescu, and Shafrir \cite{BrezisMironescuShafrir2016} introduced the `directed distance' 
\begin{equation}
    \label{eq:directeddistgen}
    \distdir_{p}(\E_{d_1}^{(n,p)},\E_{d_2}^{(n,p)}) := \sup_{f\in\E_{d_1}^{(n,p)}}
    \inf_{g\in\E_{d_2}^{(n,p)}} \| f - g \|_{\dot W^{n/p,p}(\Sph^n;\Sph^n)} \,,
\end{equation}
extending \eqref{eq:directed-distance} to more general Sobolev spaces $W^{n/p,p}(\Sph^n;\Sph^n)$ with $1\leq p<\infty$. Moreover, $\E_d^{(n,p)}$ consists of maps in $W^{n/p,p}(\Sph^n;\Sph^n)$ with degree $d$ (which is well-defined in $W^{n/p,p}(\Sph^n,\Sph^n)\subset VMO(\Sph^n;\Sph^n)$). In their Open Problem 1 they asked whether $\distdir_{p}(\E_{d_1}^{(n,p)},\E_{d_2}^{(n,p)}) = \distdir_{p}(\E_{d_2}^{(n,p)},\E_{d_1}^{(n,p)})$ or, even better, whether it only depends on $|d_1-d_2|$. In their Open Problem 2 they asked whether one has a lower bound by some constant times $|d_1-d_2|^{1/p}$. Since the directed distance is bounded above by the corresponding Hausdorff distance, \cite[Theorem~1.7(1)]{BrezisMironescuShafrir2016} gives
\begin{equation}\label{eq:BMS-upper}
  \distdir_{p}(\E_{d_1}^{(n,p)},\E_{d_2}^{(n,p)})
  \le C_{p,n}|d_1-d_2|^{1/p}.
\end{equation}
For a fixed source degree, \cite[Proposition~7.8]{BrezisMironescuShafrir2016} gives
\begin{equation}\label{eq:BMS-fixed-source}
  \distdir_{p}(\E_{d_1}^{(n,p)},\E_{d_2}^{(n,p)})
  \ge c_{p,n,d_1}|d_1-d_2|^{1/p},
\end{equation}
where the constant may depend on $d_1$. When $1\le p\le n+1$ and $d_1d_2\le0$, they proved the uniform estimate
\begin{equation}\label{eq:BMS-opposite-sign}
  \distdir_{p}(\E_{d_1}^{(n,p)},\E_{d_2}^{(n,p)})
  \ge c_{p,n}|d_1-d_2|^{1/p} \,,
\end{equation}
see \cite[Proposition~7.5]{BrezisMironescuShafrir2016}.  For the case $p=n$ that we are interested in, these estimates give the correct order of growth in every fixed-source regime and uniformly for opposite-sign degrees, but do not identify the sharp constant or imply symmetry for arbitrary pairs.

As we have already mentioned, the analogous problem for $n=1=p$ was completely settled by \cite[Theorem~1.4(1)]{BrezisMironescuShafrir2016}. Optimal results for $n=1$ and $1<p<\infty$ appear in \cite[Theorem~1.1 and Corollary~1.2]{Shafrir2018}.

In passing, we note that if one replaces the outer supremum in \eqref{eq:directeddistgen} by an infimum, one obtains zero; see \cite[Theorem~4(1)]{BrezisMironescuShafrir2016}. For related results in sub-critical Sobolev spaces we refer to \cite[Theorem~3.4]{LeviShafrir2014} and \cite[Theorem~1 and Remark~2.1]{RubinsteinShafrir2007}.

A recent development, which is decisive for our proof, is the heterotopic energy of Detaille and Van Schaftingen.  Their formula identifies the least $W^{1,n}$ energy of maps in one homotopy class converging almost everywhere to a map in another class; for sphere targets, the defect is exactly $\kappa_n$ times the difference of degree \cite[Theorem~1.1 and Corollary~1.2]{DetailleVanSchaftingen}.  Their bubbling theorem is stronger than the total-energy inequality: it locates the topological defect in atoms of the limiting energy measure \cite[Theorem~5.1]{DetailleVanSchaftingen}. In the introduction of their paper, they note explicitly that the heterotopic energy has ``no apparent formal mathematical connection'' with the earlier directed-distance functionals. A by-product of our paper is to develop such a connection.


\subsection*{Outline of the proof}

To prove Theorem \ref{thm:main} we will prove an upper and a lower bound on the left side in \eqref{eq:main}.

The proof of the upper bound is much easier than that of the lower bound and is presented in Section \ref{sec:upper}. Similar arguments have appeared before in the literature. The idea is to locally insert a bubble. The minimal energy carried by one unit of Brouwer degree is $\kappa_n$, as is made precise in \Cref{prop:bubble-energy}. The intuition is that therefore it costs an energy of at most $|d_2-d_1|$ times $\kappa_n$ to change the degree $d_1$ of a given map $f$ to $d_2$ of some other map.

The proof of the lower bound takes up Sections \ref{sec:bubbling}, \ref{sec:pinning-construction} and 
\ref{sec:lower}. In Section \ref{sec:pinning-construction} we first construct a finite collection of rotations, all arbitrarily close to the identity, such that the orbit of every $p\in\Sph^n$ affinely spans $\R^{n+1}$.  Applying these rotations to a fixed smooth map $f_0$ in rapidly alternating phases gives sphere-valued maps $F_j$ of the same degree as $f_0$. The maps $F_j$ are `hard to approximate' in the sense that if a sequence $g_j$ of sphere-valued maps satisfies a uniform $L^n$-bound on $\D(F_j-g_j)$, compactness and periodic averaging force $F_j-g_j\to0$ in $L^n$; this is \Cref{prop:pinning}. We call this mechanism oscillatory pinning.

Now for near-minimizers $g_j$ corresponding to $F_j$, we remove the oscillation from the difference and radially normalize:
\[
  u_j:=\frac{f_0-(F_j-g_j)}{|f_0-(F_j-g_j)|}.
\]
Because $F_j$ is uniformly close to $f_0$, the denominator stays away from zero and $u_j$ has the degree of $g_j$.  Pinning gives $u_j\to f_0$ in $L^n$. Thus, we are in the situation where all the maps $u_j$ have a fixed degree $d_2$, while their $L^n$ limit has degree $d_1$. A special case of the bubbling theorem of Detaille and Van Schaftingen (see \Cref{prop:atomic-defect}) forces at least $\kappa_n \, |d_1-d_2|$ of the limiting energy of $u_j$ into atoms.  The derivative of the fixed smooth background $f_0$ contributes only an absolutely continuous measure.  An atom-by-atom comparison therefore transfers the entire topological cost to the limiting measure of $|\D(F_j-g_j)|^n$.  Comparing atoms, rather than total $L^n$ energies, is the step that works uniformly in every dimension. The details of this argument appear in Section \ref{sec:lower}.


\section{Degree and homotopy}\label{sec:preliminaries}

In this section we collect some well-known facts that we need in what follows. Throughout, derivatives are tangential derivatives on $\Sph^n$, and we often write $\D u$ for $\D_\tau u$.

Let $(\tau_1,\dots,\tau_n)$ be a positively oriented local orthonormal frame of $T\Sph^n$. For $u\in W^{1,n}(\Sph^n;\Sph^n)$, define
\begin{equation}\label{eq:weak-jac}
  \Jac u:=\det\bigl(u,\partial_{\tau_1}u,\dots,\partial_{\tau_n}u\bigr).
\end{equation}
The definition is independent of the positively oriented orthonormal frame. Since $|u|=1$ a.e., one has
\[
  u\cdot\partial_{\tau_i}u=0
  \qquad\text{a.e. for }i=1,\dots,n.
\]
Consequently, if $s_1,\dots,s_n$ are the singular values of $\D u$, then
\begin{equation}
    \label{eq:jacobianbound}
      |\Jac u|=\prod_{i=1}^n s_i
  \le\left(\frac{s_1^2+\cdots+s_n^2}{n}\right)^{n/2}
  =n^{-n/2}|\D u|^n.
\end{equation}
Thus $\Jac u\in L^1$ and
\begin{equation}\label{eq:degree-def}
  \deg u:=\frac1{|\Sph^n|}\int_{\Sph^n}\Jac u\,\dd\sigma
\end{equation}
is well-defined.

For smooth functions $u$, $\deg u$ coincides with one of the standard definitions of the Brouwer degree and, in particular, it is an integer. In \Cref{prop:density-classification} we recall that $C^\infty(\Sph^n;\Sph^n)$ is strongly dense in $W^{1,n}(\Sph^n;\Sph^n)$ and that the degree is continuous under strong $W^{1,n}$ convergence. It follows from these facts that $\deg u\in\Z$ for all $u\in W^{1,n}(\Sph^n;\Sph^n)$. For further discussion of the degree and its properties, we refer to \cite{BrezisNirenberg1995,BrezisMironescuShafrir2016} and references therein.

\begin{Proposition}[Critical density and classification]\label[Proposition]{prop:density-classification}
The following hold:
\begin{enumerate}[label=\textup{(\roman*)}]
\item $C^\infty(\Sph^n;\Sph^n)$ is strongly dense in $W^{1,n}(\Sph^n;\Sph^n)$;
\item degree is continuous under strong $W^{1,n}$ convergence;
\item for $u,v\in W^{1,n}(\Sph^n;\Sph^n)$, we have $\deg u=\deg v$ if and only if $u$ and $v$ lie in the same $W^{1,n}$ path component;
\item for every $u\in W^{1,n}(\Sph^n;\Sph^n)$,
\begin{equation}\label{eq:degree-energy}
  \int_{\Sph^n}|\D u|^n\,\dd\sigma
  \ge\kappa_n|\deg u|.
\end{equation}
\end{enumerate}
\end{Proposition}

\begin{proof}
Assertion (i) for $n=2$ is due to \cite{SchoenUhlenbeck1983}. The same proof extends to $n\ge 3$, see \cite[Proposition~A.2]{BrezisLi2001} and also \cite[Lemma~2.1]{BrezisMironescuShafrir2016}.

Assertion (ii) is \cite[Theorem~2.3]{BrezisMironescuShafrir2016}.  It also follows directly from the Jacobian formula for the degree, the multilinearity of the determinant, H\"older's inequality, and dominated convergence.

For (iii), \cite[Proposition~0.3]{BrezisLi2001} identifies the $W^{1,n}$ homotopy classes with the ordinary continuous homotopy classes.  By the classical Hopf classification, continuous maps $\Sph^n\to\Sph^n$ are homotopic if and only if they have the same Brouwer degree.  Together with (ii), this proves the equivalence.

For (iv), we observe that, by \eqref{eq:degree-def} and the Jacobian bound \eqref{eq:jacobianbound},
\[
  |\Sph^n|\, |\deg u|
  =\left|\int_{\Sph^n}\Jac u\,\dd\sigma\right|
  \le\int_{\Sph^n}|\Jac u|\,\dd\sigma
  \le n^{-n/2}\int_{\Sph^n}|\D u|^n\,\dd\sigma.
  \qedhere
\]
\end{proof}

We will often make use of the radial projection map $\Pi$, defined by
\[
  \Pi(z):=\frac{z}{|z|},\qquad z\in\R^{n+1}\setminus\{0\}.
\]
Its derivative is
\begin{equation}\label{eq:radial-derivative}
  \D\Pi(z)[A]
  =\frac1{|z|}\left(I-\frac z{|z|}\otimes\frac z{|z|}\right)A,
\end{equation}
so
\begin{equation}\label{eq:radial-bound}
  |\D\Pi(z)[A]|\le\frac{|A|}{|z|}.
\end{equation}

\begin{lemma}[Uniformly close maps have the same degree]\label[lemma]{lem:close-degree}
Let $u,v\in W^{1,n}(\Sph^n;\Sph^n)$.  If
\[
  \esssup_{\Sph^n}|u-v|<2,
\]
then $u$ and $v$ are homotopic in $W^{1,n}$.  In particular, $\deg u=\deg v$.
\end{lemma}

\begin{proof}
Set
\[
  w_t:=(1-t)u+tv
  \qquad\text{and}\qquad
  H_t:=\Pi(w_t) \,.
\]
Since $|u|=|v|=1$, setting $c:= \esssup_{\Sph^n}|u-v|$ gives
\[
  |w_t|^2
  =1-t(1-t)|u-v|^2
  \ge 1- c^2/4 > 0 \,,
\]
so $H_t$ is well-defined.  On the set $\{|z|^2\ge 1- c^2/4 \}$, both $\D\Pi$ and $\D^2\Pi$ are bounded by constants depending only on $c$.  Therefore, for $s,t\in[0,1]$,
\[
  \|H_t-H_s\|_{L^n}
  \le C_c|t-s|\|u-v\|_{L^n}
\]
and, using the chain rule,
\[
  \|\D H_t-\D H_s\|_{L^n}
  \le C_c|t-s|\bigl(\|\D u\|_{L^n}+\|\D v\|_{L^n}\bigr).
\]
Thus $t\mapsto H_t$ is a continuous path in $W^{1,n}(\Sph^n;\Sph^n)$ from $u$ to $v$. In particular, \Cref{prop:density-classification} gives $\deg u=\deg v$.
\end{proof}

\begin{lemma}[Excision and additivity of degree]\label[lemma]{lem:degree-excision}
Let $u,v\in C(\Sph^n;\Sph^n)$, and suppose that $u=v$ on the complement of open balls $B_1,\dots,B_I$ whose closures are pairwise disjoint.  For each $i$, form a map $U_i:\Sph^n\to\Sph^n$ by taking $u$ on one copy of $\overline B_i$ and $v$ on a second copy of $\overline B_i$ with the opposite orientation.  The two pieces match on the boundary because $u=v$ there.  Put
\[
  q_i:=\deg U_i .
\]
With this orientation convention,
\begin{equation}\label{eq:degree-excision}
  \deg u-\deg v=\sum_{i=1}^I q_i.
\end{equation}
\end{lemma}

\begin{proof}
This is the standard additivity, or excision, property of the Brouwer degree.  To see the sign convention concretely, first assume that the maps are smooth and that the common boundary values are regular enough so that one may compute degrees by signed counting of preimages of a regular value $y$ that is not hit on the boundaries $\partial B_i$.  The preimages of $y$ outside the balls are the same for $u$ and $v$ and therefore cancel in $\deg u-\deg v$.  Inside $B_i$, the remaining signed count is exactly the degree of the glued map $U_i$, because the second copy carries the opposite orientation.  Summing over the balls gives \eqref{eq:degree-excision}.  The general continuous case follows by uniform approximation and homotopy invariance of the degree.
\end{proof}


\section{The sharp bubble upper bound}\label{sec:upper}

Our goal in this section is to prove the upper bound in Theorem \ref{thm:main}, viz.,
$$
\sup_{f\in\E_{d_1}^{(n)}}\inf_{g\in\E_{d_2}^{(n)}}
  \int_{\Sph^n}|\D_\tau f-\D_\tau g|^n\,\dd\sigma
  \leq \kappa_n \, |d_1-d_2| \,.
$$
This will be the content of \Cref{prop:upper} below.

The following proposition is, essentially, \cite[Proposition~4.10]{DetailleVanSchaftingen}. Since we need a slight refinement and since the method is interesting, we include a complete proof. We let $B^n$ denote the open unit ball in $\R^n$.

\begin{Proposition}[Energy of a degree bubble]\label[Proposition]{prop:bubble-energy}
Let $q\in\Z$ and $p\in\Sph^n$.  Among maps $b\in W^{1,n}(B^n;\Sph^n)$ whose trace is the constant $p$ and whose relative degree is $q$,
\begin{equation}\label{eq:bubble-energy}
  \inf_b\int_{B^n}|\D b|^n\,\dd x
  =\kappa_n|q|.
\end{equation}
The same infimum is obtained if $b$ is required to be smooth and constant in a neighborhood of $\partial B^n$.
\end{Proposition}

Here and below the relative degree is understood as follows.  If $b\in W^{1,n}(B^n;\Sph^n)$ has trace equal to $p$, extend it by the constant value $p$ to $\R^n\setminus B^n$.  This gives a map $\widetilde b\in W^{1,n}(\R^n;\Sph^n)$ with finite $n$-energy.  Compactifying $\R^n$ by adding the point at infinity, equivalently precomposing $\widetilde b$ with an orientation-preserving stereographic chart from $\Sph^n\setminus\{N\}$ onto $\R^n$ and setting the value at $N$ equal to $p$, gives a map $\widehat b\in W^{1,n}(\Sph^n;\Sph^n)$.  We define the relative degree of $b$ to be $\deg \widehat b$.  This is independent of the chosen orientation-preserving chart.

\begin{proof}
Let $b$ be admissible and let $\widehat b$ be its compactification as above. Since stereographic compactification is conformal and the $n$-energy is conformally invariant,
\[
  \int_{\Sph^n}|\D \widehat b|^n\,\dd\sigma
  =\int_{\R^n}|\D \widetilde b|^n\,\dd x
  =\int_{B^n}|\D b|^n\,\dd x .
\]
By \Cref{prop:density-classification}(iv) the left side is bounded from below by $\kappa_n|\deg\widehat b|=\kappa_n|\deg b| = \kappa_n |q|$, so we obtain that for any admissible $b$ we have
\[
  \int_{B^n}|\D b|^n\,\dd x\ge \kappa_n|q|.
\]

For the reverse inequality, let $\Phi:\R^n\to\Sph^n\setminus\{p\}$ be inverse stereographic projection, chosen so that its compactification has degree $1$ and so that $\Phi(x)\to p$ as $|x|\to\infty$.  It is conformal, it has $n$ non-zero singular values and they are all equal to $2/(1+|x|^2)$. Hence
\begin{equation}\label{eq:stereographic-energy}
  |\D\Phi(x)|^n
  =n^{n/2}\left(\frac{2}{1+|x|^2}\right)^n,
  \qquad
  \int_{\R^n}|\D\Phi|^n\,\dd x=\kappa_n.
\end{equation}
Choose $\chi\in C^\infty([0,\infty);[0,1])$ with $\chi=1$ on $[0,1]$ and $\chi=0$ on $[2,\infty)$. Since $|\Phi(x)-p|=O(|x|^{-1})$ and $|\D\Phi(x)|=O(|x|^{-2})$, we may choose $R$ so large that $|\Phi(x)-p|\le1/2$ for $|x|\ge R$.  For $x\in B_{2R}\setminus B_R$, the vector $p+\chi(|x|/R)(\Phi(x)-p)$ has norm at least $1/2$. Thus, we can set
\[
  \Phi_R(x):=\Pi\!\left(p+\chi(|x|/R)(\Phi(x)-p)\right)
  \qquad\text{for all}\ x\in\R^n \,.
\]
Then $\Phi_R=\Phi$ on $B_R$, $\Phi_R=p$ outside $B_{2R}$ and, on $B_{2R}\setminus B_R$,
\[
  |\D\Phi_R(x)|\le CR^{-2} \,.
\]
Consequently,
\[
  \int_{B_{2R}\setminus B_R}|\D\Phi_R|^n\,\dd x
  \le CR^nR^{-2n}=CR^{-n}\to0 \,.
\]
Together with \eqref{eq:stereographic-energy}, this gives
\[
  \int_{\R^n}|\D\Phi_R|^n\,\dd x\to\kappa_n.
\]
Moreover, the homotopy
\[
  \Pi\!\left(p+((1-t)+t\chi(|x|/R))(\Phi(x)-p)\right),
  \qquad0\le t\le1,
\]
is well-defined for large $R$, is equal to $p$ near infinity, and connects $\Phi$ to $\Phi_R$ after compactification.  Hence the compactification of $\Phi_R$ has degree $1$.

Since $\Phi_R=p$ on $\R^n\setminus B_{2R}$, the restriction of $\Phi_R$ to $B_{3R}$ is constant near $\partial B_{3R}$.  After rescaling $B_{3R}$ to $B^n$, conformal invariance under dilations gives a smooth degree-one map on $B^n$, constant near $\partial B^n$, whose energy tends to $\kappa_n$.  Composing with an orientation-reversing target isometry that fixes $p$ gives degree $-1$ with the same boundary value and the same energy.

For $q\ne0$, choose $|q|$ pairwise disjoint balls compactly contained in $B^n$, place in each a rescaled copy of the degree-$\operatorname{sgn}q$ construction, and set the map equal to $p$ elsewhere.  The map is smooth and constant near $\partial B^n$; by \Cref{lem:degree-excision}, its relative degree is $q$, and its energy is the sum of the energies of the copies.  Letting the error of each copy tend to zero gives the upper bound $\kappa_n|q|$.  The case $q=0$ is given by the constant map.
\end{proof}

The following approximation result complements \Cref{prop:density-classification}(i) and is a special case of \cite[Lemma 2.2]{BrezisMironescuShafrir2016}. We include the simple proof to make our paper selfcontained.

\begin{lemma}[Local flattening]\label[lemma]{lem:flattening}
Let $f\in C^\infty(\Sph^n;\Sph^n)$ and $a\in\Sph^n$.  For all sufficiently small $r>0$, there is $\widetilde f_r\in C^\infty(\Sph^n;\Sph^n)$ such that
\begin{enumerate}[label=\textup{(\roman*)}]
\item $\widetilde f_r=f(a)$ on $B_r(a)$;
\item $\widetilde f_r=f$ on $\Sph^n\setminus B_{2r}(a)$;
\item $\deg\widetilde f_r=\deg f$;
\item $\|\D(\widetilde f_r-f)\|_{L^n}\to0$ as $r\downarrow0$.
\end{enumerate}
\end{lemma}

\begin{proof}
Choose a smooth cutoff $\eta_r$ with
\[
  \begin{gathered}
    0\le\eta_r\le1,\qquad \eta_r=0\text{ on }B_r(a),\\
    \eta_r=1\text{ outside }B_{2r}(a),\qquad |\D\eta_r|\le C/r.
  \end{gathered}
\]
Set $p:=f(a)$ and $z_r:=p+\eta_r(f-p)$.  Since $f$ is smooth, $|f(x)-p|\le Cr$ on $B_{2r}(a)$; hence $|z_r|\ge1-Cr\geq 1/2$ for all sufficiently small $r>0$.  Define $\widetilde f_r:=\Pi(z_r)$.  Properties (i) and (ii) are immediate.

For (iii), set
\[
  H_t(x):=\Pi\bigl(p+((1-t)+t\eta_r(x))(f(x)-p)\bigr),
  \qquad0\le t\le1.
\]
The vector inside $\Pi$ has norm at least $1-Cr\geq 1/2$, so $H$ is a smooth homotopy with $H_0=f$, $H_1=\widetilde f_r$, and $H_t=f$ outside $B_{2r}(a)$.  Therefore, either by classical theory or by \Cref{prop:density-classification}, the degree is unchanged.

On $B_{2r}(a)$,
\[
  \D z_r
  =\eta_r\D f+(f-p)\otimes\D\eta_r,
\]
and therefore
\[
  |\D z_r|
  \le|\D f|+|f-p||\D\eta_r|
  \le C,
\]
with $C$ independent of $r$.  By \eqref{eq:radial-bound}, $|\D\widetilde f_r|\le2C$ there.  Hence
\[
  |\D(\widetilde f_r-f)|
  \le|\D\widetilde f_r|+|\D f|
  \le C
  \quad\text{on }B_{2r}(a).
\]
Outside $B_{2r}(a)$ the two maps agree, so
\[
  \int_{\Sph^n}|\D(\widetilde f_r-f)|^n\,\dd\sigma
  \le C|B_{2r}(a)|\to0.
  \qedhere
\]
\end{proof}

\begin{Proposition}[Sharp upper bound]\label[Proposition]{prop:upper}
For every $n\ge2$, $d_1,d_2\in\Z$, and $f\in\E_{d_1}^{(n)}$,
\begin{equation}\label{eq:upper-pointwise}
  \inf_{g\in\E_{d_2}^{(n)}}
  \int_{\Sph^n}|\D f-\D g|^n\,\dd\sigma
  \le\kappa_n \, |d_1-d_2|.
\end{equation}
\end{Proposition}

\begin{proof}
Assume first that $f$ is smooth.  Fix $a\in\Sph^n$ and $p:=f(a)$. By \Cref{lem:flattening}, for all sufficiently small $r>0$ there is $\widetilde f_r\in C^\infty(\Sph^n;\Sph^n)$ with $\deg\widetilde f_r = d_1$, with $\widetilde f_r = p$ on $B_r(a)$ and $\widetilde f_r = f$ on $\Sph^n\setminus B_{2r}(a)$ and such that $\|\D(f-\widetilde f_r)\|_{L^n}\to0$.  Choose an orientation-preserving conformal diffeomorphism $\psi:B_r(a)\to B^n$.  

Fix $\epsilon>0$. Conformal invariance of the $n$-energy and \Cref{prop:bubble-energy} imply that there is
\[
  b_0\in C^\infty(\overline{B^n};\Sph^n)
\]
that is constant equal to $p$ near $\partial B^n$, has relative degree $d_2-d_1$, and satisfies
\[
  \int_{B^n}|\D b_0|^n\,\dd x
  \le\kappa_n \, |d_1-d_2|+\varepsilon.
\]
Set $b:=b_0\circ\psi$.  Then $b$ is smooth, is constant equal to $p$ near $\partial B_r(a)$, has relative degree $d_2-d_1$, and
\[
  \int_{B_r(a)}|\D b|^n\,\dd\sigma
  \le\kappa_n \, |d_1-d_2|+\varepsilon.
\]
Define $G:=b$ on $B_r(a)$ and $G:=\widetilde f_r$ outside.  This piecewise map is continuous, and the one-ball case of \Cref{lem:degree-excision} gives $\deg G=\deg\widetilde f_r+d_2-d_1=d_2$.  Moreover,
\[
  \|\D(\widetilde f_r-G)\|_{L^n}^n
  =\int_{B_r(a)}|\D b|^n\,\dd\sigma.
\]
The $L^n$ triangle inequality yields
\[
  \|\D(f-G)\|_{L^n}
  \le\|\D(f-\widetilde f_r)\|_{L^n}
    +(\kappa_n \, |d_1-d_2|+\varepsilon)^{1/n}.
\]
It follows that
$$
\inf_{g\in\E_{d_2}^{(n)}}
  \int_{\Sph^n}|\D f-\D g|^n\,\dd\sigma
  \le \left( \|\D(f-\widetilde f_r)\|_{L^n} + (\kappa_n \, |d_1-d_2| + \varepsilon)^{1/n}\right)^n.
$$
Letting $r\downarrow0$ and then $\varepsilon\downarrow0$, we obtain the claimed inequality \eqref{eq:upper-pointwise} for smooth $f$.

For general $f\in\E_{d_1}^{(n)}$, choose smooth $f_j$ with $f_j\to f$ strongly in $W^{1,n}$. In particular, for all large $j$, $\deg f_j=d_1$.  For such $j$, choose $g_j\in\E_{d_2}^{(n)}$ such that
\[
  \|\D(f_j-g_j)\|_{L^n}
  \le(\kappa_n \, |d_1-d_2|+1/j)^{1/n}.
\]
Then
\[
  \inf_{g\in\E_{d_2}^{(n)}}\|\D(f-g)\|_{L^n}
  \le\|\D(f-f_j)\|_{L^n}+(\kappa_n \, |d_1-d_2|+1/j)^{1/n}.
\]
Letting $j\to\infty$ proves \eqref{eq:upper-pointwise}.
\end{proof}


\section{Atomic bubbling and the topological defect}\label{sec:bubbling}

An important ingredient in our proof of the lower bound is the following specialization of the bubbling theorem of Detaille and Van Schaftingen \cite{DetailleVanSchaftingen}.

\begin{Proposition}[Sphere-valued form of the DVS bubbling theorem]\label[Proposition]{prop:atomic-defect}
Let $u_j\in W^{1,n}(\Sph^n;\Sph^n)$ satisfy
\[
  \deg u_j=d_2,
  \qquad u_j\to u\text{ in }L^1(\Sph^n),
  \qquad u\in C^\infty(\Sph^n;\Sph^n),\quad\deg u=d_1.
\]
Suppose
\[
  |\D u_j|^n\,\dd\sigma\weakstar\mu
\]
as finite Radon measures.  Then there are finitely many points $a_1,\dots,a_I\in\Sph^n$ such that
\begin{equation}\label{eq:atomic-defect}
  \sum_{i=1}^I\mu(\{a_i\})
  \ge\kappa_n \, |d_1-d_2|.
\end{equation}
\end{Proposition}

\begin{proof}
Choose a fixed smooth map $v:\Sph^n\to\Sph^n$ with $\deg v=d_2$. Since $\deg u_j=\deg v=d_2$, Proposition~\ref{prop:density-classification}
gives a $W^{1,n}$-homotopy between $u_j$ and $v$. Since
$W^{1,n}(\Sph^n)\hookrightarrow \VMO(\Sph^n)$ continuously, this is also
a $\VMO$-homotopy. Therefore the homotopy-class hypothesis in
\cite[Theorem~5.1]{DetailleVanSchaftingen} is satisfied. Apply \cite[Theorem~5.1]{DetailleVanSchaftingen} with
\[
  \mathcal M=\mathcal N=\Sph^n,\qquad m=n,
  \qquad v_j=u_j.
\]
The hypotheses are exactly the $L^1$ convergence above, the common $\VMO$ homotopy class of the sequence, and the assumed weak-star convergence of the energy measures.  The theorem yields finitely many points $a_1,\dots,a_I$ such that, for every sufficiently small common radius $r>0$, there is a continuous map $w:\Sph^n\to\Sph^n$ satisfying
\begin{enumerate}[label=\textup{(\roman*)}]
\item $w$ is homotopic to $u$, and hence $\deg w=d_1$;
\item $w=v$ on $\Sph^n\setminus\bigcup_{i=1}^I B_r(a_i)$;
\item
\begin{equation}\label{eq:dvs-measure}
  \mu\ge|\D u|^n\,\dd\sigma
  +\sum_{i=1}^I
  \mathfrak E_{\mathrm{top}}^{1,n}([w,v,B_r(a_i)])\,\delta_{a_i}.
\end{equation}
\end{enumerate}
Here $[w,v,B_r(a_i)]\in\pi_n(\Sph^n)$ is the local disparity in the convention of \cite[Section~4.1]{DetailleVanSchaftingen}, that is, the homotopy class of maps from $\Sph^n$ to $\Sph^n$ that are homotopic
to a map given by $w|_{B_r(a_i)}$ on the northern hemisphere and by $v|_{B_r(a_i)}$ on the southern one. We do not need to recall the definition of the energy functional $\mathfrak E_{\mathrm{top}}^{1,n}$, because the equality \eqref{eq:local-topological-energy} below is all we need.

The balls can be chosen with pairwise disjoint closures. Let
\[
  q_i:=\deg [w,v,B_r(a_i)]\in\Z.
\]
Since $w=v$ outside the union of the balls, \Cref{lem:degree-excision} gives
\begin{equation}\label{eq:local-degree-sum}
  \sum_{i=1}^I q_i=\deg w-\deg v=d_1-d_2.
\end{equation}
By \cite[Proposition~4.10]{DetailleVanSchaftingen} (which is essentially \Cref{prop:bubble-energy}), the topological energy of a sphere-valued disparity is
\begin{equation}\label{eq:local-topological-energy}
  \mathfrak E_{\mathrm{top}}^{1,n}([w,v,B_r(a_i)])
  =\kappa_n|q_i|.
\end{equation}
Taking the mass of \eqref{eq:dvs-measure} at the points and using \eqref{eq:local-degree-sum},
\[
  \sum_{i=1}^I\mu(\{a_i\})
  \ge\kappa_n\sum_{i=1}^I|q_i|
  \ge\kappa_n\left|\sum_{i=1}^I q_i\right|
  =\kappa_n \, |d_1-d_2|.
  \qedhere
\]
\end{proof}


\section{The oscillatory pinning construction}\label{sec:pinning-construction}

Our goal in this section is to construct a sequence of maps of a given degree that are hard to approximate in the sense of Proposition \ref{prop:pinning}. The construction is analogous in spirit to the one-dimensional zig-zag constructions in \cite{BrezisMironescuShafrir2016,BrezisMironescuShafrir2018} and \cite{Shafrir2018}, but the higher-dimensional mechanism is different: we rapidly alternate among finitely many small target rotations whose orbits affinely span the ambient space. For this reason we will refer to the mechanism as oscillatory pinning.


\subsection{A uniform affine-spanning family of small rotations}\label{sec:rotations}

Set $m:=n+1$ and let $\so(m)$ denote the set of all traceless, skew-symmetric, real $m\times m$-matrices. Let $(e_a)_{a=1}^m$ be the standard basis of $\R^m$. For $1\le a<b\le m$, define $A_{ab}\in\so(m)$ by
\begin{equation}\label{eq:Aab}
  A_{ab}e_a :=e_b,\qquad A_{ab}e_b:=-e_a,
\end{equation}
and
$$
A_{ab}e_c=0
\quad\text{if}\ c\notin\{a,b\} \,.
$$
Fix $\theta\in(0,\pi)$ and put
\begin{equation}\label{eq:rotation-family}
  Q_0:=I,\qquad Q_{ab}^{\pm}:=\exp(\pm\theta A_{ab}),
\end{equation}
with
\[
  \mathcal Q_\theta
  :=\{Q_0\}\cup\{Q_{ab}^{+},Q_{ab}^{-}:1\le a<b\le m\}.
\]
We emphasize that the matrices in $\mathcal Q_\theta$ are orthogonal. It will be important later that the angle $\theta$ may be arbitrarily small, so the whole family may be chosen uniformly close to the identity.

\begin{lemma}[Uniform affine spanning]\label[lemma]{lem:affine-spanning}
For every $p\in\Sph^n$, the set $\{Qp:Q\in\mathcal Q_\theta\}$ affinely spans $\R^{n+1}$.
\end{lemma}

\begin{proof}
Fix $p\in\Sph^n$ and suppose $\xi\in\R^m$ is a linear functional that is constant on $\{Qp:Q\in\mathcal Q_\theta\}$. We need to show that $\xi=0$, because this means that the set $\{Qp:Q\in\mathcal Q_\theta\}$ is not contained in any proper affine hyperplane, so it is affine spanning.

Since $I\in\mathcal Q_\theta$, the constancy of $\xi$ implies that
\begin{equation}\label{eq:constant-functional}
  \xi\cdot Qp=\xi\cdot p
  \qquad \text{for all}\ Q\in\mathcal Q_\theta \,.
\end{equation}
Subtracting the equations for $Q_{ab}^{+}$ and $Q_{ab}^{-}$ gives
\[
  0=\xi\cdot(Q_{ab}^{+}-Q_{ab}^{-})p
  =2\sin\theta\,\xi\cdot A_{ab}p.
\]
The vectors $A_{ab}p$ span $p^\perp$: the matrices $A_{ab}$ span $\so(m)$, and for $z\perp p$ the skew matrix $z\otimes p-p\otimes z$ sends $p$ to $z$.  Hence $\xi=\lambda p$ for some $\lambda\in\R$.

Note that 
\[
  Q_{ab}^{+}+Q_{ab}^{-}-2I
  =2(\cos\theta-1)P_{ab},
\]
where $P_{ab}$ projects onto $\operatorname{span}\{e_a,e_b\}$. Inserting this into \eqref{eq:constant-functional} gives
\[
  0=2(\cos\theta-1)\lambda \, p\cdot P_{ab}p \,.
\]
Since $p\neq 0$, there is a pair $(a,b)$ with $p_a^2+p_b^2 = p\cdot P_{ab}p>0$. This, together with $\cos\theta\neq 1$, leads to the desired conclusion $\lambda=0$ and therefore $\xi=0$.
\end{proof}


\subsection{Periodic averaging on the sphere}\label{sec:averaging}

In this subsection we derive a simple homogenization result on $\Sph^n$ that will be useful in the next subsection.

\begin{lemma}[One-dimensional periodic averaging]\label[lemma]{lem:one-d-averaging}
Let $a\in L^\infty(\R)$ be one-periodic and $\bar a:=\int_0^1a(s)\,\dd s$.  Then
\[
  a(jt)\weakstar\bar a
  \quad\text{in }L^\infty((-1,1))
\]
as $j\to\infty$ through positive integers.
\end{lemma}

\begin{proof}
Replacing $a$ by $a-\bar a$, we may assume that $a$ has mean-value zero. Let
\[
  A(t):=\int_0^t a(s)\,\dd s.
\]
Since $a$ is one-periodic and has mean zero, $A$ is one-periodic. Clearly, it is bounded. For $\varphi\in C_c^1((-1,1))$, integration by parts gives
\[
  \int_{-1}^1a(jt)\varphi(t)\,\dd t
  =-\frac1j\int_{-1}^1A(jt)\varphi'(t)\,\dd t\to0.
\]
Uniform $L^\infty$ boundedness and density of $C_c^1$ in $L^1$ prove the weak-star convergence.
\end{proof}

\begin{lemma}[Periodic averaging on $\Sph^n$]\label[lemma]{lem:sphere-averaging}
Under the assumptions of \Cref{lem:one-d-averaging},
\begin{equation}\label{eq:sphere-average}
  a(jx_{n+1})\weakstar\bar a
  \quad\text{in }L^\infty(\Sph^n).
\end{equation}
\end{lemma}

\begin{proof}
We parametrize $x=(\sqrt{1-t^2}\,\omega,t)$ with $\omega\in\Sph^{n-1}$ and $-1<t<1$. For $\Phi\in L^1(\Sph^n)$, we have
\[
  \int_{\Sph^n}\Phi(x)a(jx_{n+1})\,\dd\sigma(x)
  =\int_{-1}^1a(jt)G_\Phi(t)\,\dd t
\]
with
\[
  G_\Phi(t)
  : = (1-t^2)^{(n-2)/2}
  \int_{\Sph^{n-1}}\Phi(\sqrt{1-t^2}\,\omega,t)\,\dd\sigma_{n-1}(\omega).
\]
Moreover,
\[
  \|G_\Phi\|_{L^1(-1,1)}
  \le\|\Phi\|_{L^1(\Sph^n)}
\]
and
\[
  \int_{-1}^1G_\Phi(t)\,\dd t
  =\int_{\Sph^n}\Phi\,\dd\sigma.
\]
Thus, \Cref{lem:one-d-averaging} gives the claimed weak-star convergence.
\end{proof}


\subsection{Oscillatory pinning}\label{sec:pinning}

We recall that the family $\mathcal Q_\theta$ of orthogonal matrices was defined in Subsection \ref{sec:rotations}.

\begin{lemma}[Smooth periodic rotation profile]\label[lemma]{lem:profile}
Fix $\rho>0$.  There are a nonzero sufficiently small angle $\theta$, a smooth one-periodic map $B:\R\to\so(n+1)$, and pairwise disjoint intervals $I_Q\subset(0,1)$ of positive length, indexed by $Q\in\mathcal Q_\theta$, such that
\begin{enumerate}[label=\textup{(\roman*)}]
\item $\exp(B(t))=Q$ for $t\in I_Q$;
\item $|\exp(B(t))p-p|\le\rho$ for every $t\in\R$ and $p\in\Sph^n$.
\end{enumerate}
\end{lemma}

\begin{proof}
Choose
\[
  0<\theta<\min\{\pi,\log(1+\rho)\}.
\]
For $Q\in\mathcal Q_\theta$, define
\[
  R_{Q_0}:=0,
  \qquad
  R_{Q_{ab}^{\pm}}:=\pm\theta A_{ab},
\]
so that $\exp(R_Q)=Q$.  Let
\[
  \mathcal R_\theta
  :=\{0\}\cup\{\pm\theta A_{ab}:1\le a<b\le n+1\}.
\]
If $A\in\operatorname{conv}\mathcal R_\theta$, then $\|A\|_{\mathrm{op}}\le\theta$.  Consequently,
\[
  \|\exp(A)-I\|_{\mathrm{op}}
  \le e^{\|A\|_{\mathrm{op}}}-1
  \le e^\theta-1<\rho.
\]
Construct a smooth closed curve
\[
  B:\R/\Z\longrightarrow\operatorname{conv}\mathcal R_\theta
\]
which visits every vertex $R_Q$ and is constant equal to $R_Q$ on a nondegenerate interval $I_Q\subset(0,1)$.  This can be done by joining consecutive vertices by line segments parametrized with smooth cutoff functions that are constant near their endpoints.  Start and finish at $0$, with $B=0$ near both $0$ and $1$, so that the periodic extension is smooth.  Then $\exp(B(t))=Q$ on $I_Q$, and for every $p\in\Sph^n$,
\[
  |\exp(B(t))p-p|
  \le\|\exp(B(t))-I\|_{\mathrm{op}}
  <\rho.
  \qedhere
\]
\end{proof}

The following is the main result of this section.

\begin{Proposition}[Oscillatory pinning]\label[Proposition]{prop:pinning}
Let $f_0\in C^\infty(\Sph^n;\Sph^n)$ and $\rho>0$.  There is a sequence $F_j\in C^\infty(\Sph^n;\Sph^n)$ such that
\begin{equation}\label{eq:F-properties}
  \deg F_j=\deg f_0,
  \qquad\|F_j-f_0\|_{L^\infty}\le\rho,
\end{equation}
and with the following property: if $g_j\in W^{1,n}(\Sph^n;\Sph^n)$ and
\begin{equation}\label{eq:bounded-difference}
  \sup_j\int_{\Sph^n}|\D(F_j-g_j)|^n\,\dd\sigma<\infty,
\end{equation}
then
\begin{equation}\label{eq:pinning-convergence}
  F_j-g_j\to0
  \quad\text{strongly in }L^n(\Sph^n;\R^{n+1}).
\end{equation}
\end{Proposition}

\begin{proof}
Let $\theta$ and $B$ be as in \Cref{lem:profile}, and define
\begin{equation}\label{eq:Fj-def}
  F_j(x):=\exp(B(jx_{n+1}))f_0(x).
\end{equation}
The claimed $L^\infty$ estimate on $F_j-f_0$ follows from \Cref{lem:profile}(ii).  The maps
\[
  H_s(x):=\exp(sB(jx_{n+1}))f_0(x),
  \qquad0\le s\le1,
\]
form a smooth homotopy from $f_0$ to $F_j$, so the degrees of $F_j$ and $f_0$ agree.

Let $g_j\in W^{1,n}(\Sph^n;\Sph^n)$ satisfy \eqref{eq:bounded-difference} and put
\[
  h_j:=F_j-g_j.
\]
Because both maps are sphere-valued,
\begin{equation}\label{eq:h-bound}
  |h_j|\le2
\end{equation}
and
\begin{equation}\label{eq:sphere-identity}
  2F_j\cdot h_j=|h_j|^2
  \qquad\text{a.e.}
\end{equation}
The bounds \eqref{eq:bounded-difference} and \eqref{eq:h-bound} make $(h_j)$ bounded in $W^{1,n}(\Sph^n;\R^{n+1})$. 

We are going to show that $h_j\to 0$ strongly in $L^n(\Sph^n;\R^{n+1})$, and we are going to show this by proving that any subsequence has a further subsequence that converges strongly to zero. 

In order to simplify the presentation, we do not reflect the choice of a subsequence in the notation. Since $n\ge2$, Rellich's compactness theorem gives, after taking a subsequence,
\begin{equation}\label{eq:h-L2}
  h_j\to h\quad\text{strongly in }L^2(\Sph^n)
\end{equation}
and we need to show that $h=0$.

Fix $Q\in\mathcal Q_\theta$ and let $\chi_Q$ be the one-periodic extension of the characteristic function of $I_Q$.  Where $\chi_Q(jx_{n+1})=1$, one has $F_j=Qf_0$, and therefore
\begin{equation}\label{eq:plateau-identity}
  \chi_Q(jx_{n+1})
  \bigl(2Qf_0\cdot h_j-|h_j|^2\bigr)=0.
\end{equation}
By \Cref{lem:sphere-averaging},
\begin{equation}\label{eq:chi-average}
  \chi_Q(jx_{n+1})\weakstar |I_Q|
  \quad\text{in }L^\infty(\Sph^n).
\end{equation}
Also, \eqref{eq:h-L2} implies
\begin{equation}\label{eq:nonlinear-L1}
  2Qf_0\cdot h_j-|h_j|^2
  \to2Qf_0\cdot h-|h|^2
  \quad\text{strongly in }L^1.
\end{equation}
Indeed,
\[
  \||h_j|^2-|h|^2\|_{L^1}
  \le\|h_j-h\|_{L^2}\|h_j+h\|_{L^2}\to0.
\]

Testing \eqref{eq:plateau-identity} against an arbitrary $\varphi\in L^\infty(\Sph^n)$ and using \eqref{eq:chi-average} and \eqref{eq:nonlinear-L1}, we obtain
\[
  |I_Q|\int_{\Sph^n}\varphi\,\bigl(2Qf_0\cdot h-|h|^2\bigr)\,\dd\sigma=0.
\]
Since $|I_Q|>0$, this implies
\begin{equation}\label{eq:limit-affine-equations}
  2Qf_0\cdot h=|h|^2
  \qquad\text{a.e.~on}\ \Sph^n \,.
\end{equation}
This holds for every $Q\in\mathcal Q_\theta$.

Since $\mathcal Q_\theta$ is finite, there is a set of full measure on which \eqref{eq:limit-affine-equations} holds for all $Q\in\mathcal Q_\theta$.  At such a point $x$, if $h(x)\ne0$, then every vector $Qf_0(x)$ lies in the proper affine hyperplane
\[
  \left\{y\in\R^{n+1}:y\cdot h(x)=\frac{|h(x)|^2}{2}\right\}.
\]
This contradicts \Cref{lem:affine-spanning}.  Hence $h=0$ a.e.

We have shown that every subsequence of $(h_j)$ has a further subsequence converging to zero strongly in $L^2$.  Therefore the full sequence converges to zero in $L^2$.  Since $|h_j|\le2$,
\[
  \|h_j\|_{L^n}^n\le 2^{n-2}\|h_j\|_{L^2}^2\to0,
\]
which proves \eqref{eq:pinning-convergence}.
\end{proof}


\section{Proof of the sharp lower bound}\label{sec:lower}

In this section we complete the proof of Theorem \ref{thm:main} by proving that
\begin{equation}
    \label{eq:lowergoal}
    \sup_{f\in\E_{d_1}^{(n)}}
    \inf_{g\in\E_{d_2}^{(n)}}
    \int_{\Sph^n}|\D f-\D g|^n\,\dd\sigma \geq \kappa_n \, |d_1-d_2| \,.
\end{equation}
We fix $d_1,d_2\in\Z$ and assume $d_1\neq d_2$, for otherwise there is nothing to prove.

Our proof will depend on a parameter $0<\rho<1$ that will tend to zero at the very end of the argument. Also, we fix a smooth map $f_0:\Sph^n\to\Sph^n$ with $\deg f_0=d_1$.  Applying \Cref{prop:pinning} we obtain $F_j \in C^\infty(\Sph^n;\Sph^n)$ with
\begin{equation}\label{eq:Fj-main}
  \deg F_j=d_1,
  \qquad\|F_j-f_0\|_{L^\infty}\le\rho
\end{equation}
and with the property that \eqref{eq:bounded-difference} implies \eqref{eq:pinning-convergence}.

For each $j$, let
\begin{equation}\label{eq:Mj}
  M_j:=\inf_{g\in\E_{d_2}^{(n)}}
  \int_{\Sph^n}|\D(F_j-g)|^n\,\dd\sigma.
\end{equation}
and set
\[
L:=\liminf_{\ell\to\infty}M_\ell \,.
\]
After passing to a subsequence we may assume that
\begin{equation}\label{eq:Mj-limit}
  M_j\to L \,.
\end{equation}

By \Cref{prop:upper},
\begin{equation}\label{eq:Mj-upper}
  0\le M_j\le\kappa_n \, |d_1-d_2|\,.
\end{equation}
Choose $g_j\in\E_{d_2}^{(n)}$ such that
\begin{equation}\label{eq:near-minimizer}
  0 \le\int_{\Sph^n}|\D(F_j-g_j)|^n\,\dd\sigma - M_j \to 0 \,.
\end{equation}
It follows from \eqref{eq:Mj-upper} and \eqref{eq:near-minimizer} that $F_j-g_j$ is bounded in $W^{1,n}(\Sph^n;\R^{n+1})$ and therefore, by the construction of $F_j$ (see \Cref{prop:pinning}), we conclude that
\begin{equation}\label{eq:h-to-zero}
  F_j- g_j \to0\quad\text{strongly in }L^n(\Sph^n).
\end{equation}
After passing to a subsequence, we may assume that
\begin{equation}\label{eq:nu-limit}
  |\D (F_j-g_j)|^n\,\dd\sigma\weakstar\nu
\end{equation}
as finite Radon measures.  Testing against $1$ and using \eqref{eq:Mj-limit} and \eqref{eq:near-minimizer} gives
\begin{equation}\label{eq:nu-total}
  \nu(\Sph^n)=L.
\end{equation}


\subsection{Removing the oscillation and preserving degree}
Define
\begin{equation}\label{eq:zj}
  z_j:=f_0-F_j + g_j.
\end{equation}
Then, by \eqref{eq:Fj-main} and since $|g_j|=1$,
\begin{equation}\label{eq:z-close}
  |z_j-g_j|\le\rho,
  \qquad |z_j|\ge1-\rho>0.
\end{equation}
Define
\begin{equation}\label{eq:uj}
  u_j:=\Pi(z_j)=\frac{z_j}{|z_j|}
  \in W^{1,n}(\Sph^n;\Sph^n).
\end{equation}
We have, using $|g_j|=1$ and the first bound in \eqref{eq:z-close},
\begin{align}\label{eq:u-g-close}
  |u_j-g_j|
  & \le|u_j-z_j|+|z_j-g_j|
  =\bigl|1 - |z_j| \bigr|+|z_j-g_j| \notag \\
  & = \bigl| |g_j| - |z_j| \bigr|+|z_j-g_j| \le 2 |z_j-g_j|
  \le2\rho<2 \,.
\end{align}
Thus, by \Cref{lem:close-degree},
\begin{equation}\label{eq:u-degree}
  \deg u_j=\deg g_j=d_2.
\end{equation}

Using $|f_0|=1$, we obtain
\begin{align}\label{eq:u-f0-pointwise}
    |u_j - f_0| & \le |u_j - z_j| + |z_j - f_0| = |1-|z_j|| + |z_j - f_0| \notag \\
    & = ||f_0|-|z_j|| + |z_j - f_0| \leq 2 |z_j - f_0| \leq 2 |F_j - g_j| \,.
\end{align}
Thus, by \eqref{eq:h-to-zero},
\begin{equation}\label{eq:u-to-f0}
  u_j\to f_0\quad\text{strongly in }L^n(\Sph^n) \,.
\end{equation}

From \eqref{eq:radial-bound} and the second bound in \eqref{eq:z-close},
\begin{equation}\label{eq:Du-bound}
  |\D u_j|
  \le\frac{|\D z_j|}{1-\rho}
  =\frac{|\D f_0-\D (F_j-g_j)|}{1-\rho}.
\end{equation}
Hence $(u_j)$ is bounded in $W^{1,n}$.  Passing to another subsequence,
\begin{equation}\label{eq:mu-limit}
  |\D u_j|^n\,\dd\sigma\weakstar\mu
\end{equation}
as finite Radon measures.


\subsection{Comparison of atomic energy}
By \eqref{eq:u-degree}, \eqref{eq:u-to-f0}, and \Cref{prop:atomic-defect}, there are points $a_1,\dots,a_I \in\Sph^n$ such that
\begin{equation}\label{eq:mu-atom-lower}
  \sum_{i=1}^I\mu(\{a_i\})\ge\kappa_n \, |d_1-d_2|\,.
\end{equation}
Fix $\delta>0$. It is elementary to observe that there is a $C_\delta<\infty$ such that for any matrices $A,B$ one has
\begin{equation}\label{eq:power-inequality}
  |A+B|^n \le(1+\delta)|A|^p+C_{n,\delta}|B|^n \,.
\end{equation}
Applying this to \eqref{eq:Du-bound} with $A=-\D (F_j-g_j)$ and $B=\D f_0$ gives
\begin{equation}\label{eq:measure-density-comparison}
  |\D u_j|^n
  \le\frac1{(1-\rho)^n}
  \left((1+\delta)|\D (F_j-g_j)|^n
    +C_{n,\delta}|\D f_0|^n\right).
\end{equation}
Passing to the weak star limit we obtain the measure inequality
\begin{equation}\label{eq:measure-comparison}
  \mu\le\frac1{(1-\rho)^n}
  \left((1+\delta)\nu
    +C_{n,\delta}|\D f_0|^n\,\dd\sigma\right).
\end{equation}
By Radon regularity, this inequality holds on all Borel sets, in particular on singletons.  The last measure is absolutely continuous and has no atoms.  Therefore, for every $a\in\Sph^n$,
\[
  \mu(\{a\})
  \le\frac{1+\delta}{(1-\rho)^n}\nu(\{a\})
\]
and, after letting $\delta\downarrow0$,
\begin{equation}\label{eq:atom-comparison}
  \mu(\{a\})
  \le\frac1{(1-\rho)^n}\nu(\{a\}).
\end{equation}
Combining \eqref{eq:mu-atom-lower}, \eqref{eq:atom-comparison}, and \eqref{eq:nu-total},
\begin{align*}
  \kappa_n \, |d_1-d_2|
  &\le\sum_{i=1}^I\mu(\{a_i\})
  \le\frac1{(1-\rho)^n}\sum_{i=1}^I\nu(\{a_i\})\\
  &\le\frac1{(1-\rho)^n}\nu(\Sph^n)
  =\frac{L}{(1-\rho)^n}.
\end{align*}
Thus
\begin{equation}\label{eq:L-lower}
  L\ge(1-\rho)^n\kappa_n \, |d_1-d_2|\,. 
\end{equation}

Since every $F_j\in\E_{d_1}^{(n)}$, we have
$$
\sup_{f\in\E_{d_1}^{(n)}}
    \inf_{g\in\E_{d_2}^{(n)}}
    \int_{\Sph^n}|\D f-\D g|^n\,\dd\sigma \geq M_j \,.
$$
Taking the liminf as $j\to\infty$, we obtain
$$
\sup_{f\in\E_{d_1}^{(n)}}
    \inf_{g\in\E_{d_2}^{(n)}}
    \int_{\Sph^n}|\D f-\D g|^n\,\dd\sigma \geq L \,.
$$
This, combined with \eqref{eq:L-lower} gives
$$
\sup_{f\in\E_{d_1}^{(n)}}
    \inf_{g\in\E_{d_2}^{(n)}}
    \int_{\Sph^n}|\D f-\D g|^n\,\dd\sigma \geq (1-\rho)^n\kappa_n \, |d_1-d_2|\,. 
$$
Since $\rho>0$ is arbitrary, we obtain the claimed inequality \eqref{eq:lowergoal}.
Together with \Cref{prop:upper}, this proves \Cref{thm:main}.


\section*{Acknowledgements}
R.L.F. acknowledges partial support through US National Science Foundation grant DMS-1954995, as well as through the German Research Foundation through EXC-2111-390814868, TRR 352--Project-ID 470903074, and FR 2664/3-1.  P.I. acknowledges partial support from the US NSF CAREER grant DMS-2152401, US NSF grant DMS-2554183, a Simons Fellowship, and a Humboldt Research Fellowship for Experienced Researchers. The authors acknowledge the use of AI tools during the exploratory stage of this project.  All mathematical arguments and proofs in the final manuscript were checked and written by the authors.


\begin{thebibliography}{BMS16a}
\small\sloppy

\bibitem[BBC90]{BethuelBrezisCoron1990}
F.~Bethuel, H.~Brezis, and J.-M.~Coron,
\emph{Relaxed energies for harmonic maps},
in \emph{Variational Methods (Paris, 1988)},
Progress in Nonlinear Differential Equations and Their Applications, vol.~4,
Birkh\"auser Boston, Boston, MA, 1990, pp.~37--52.
\href{https://doi.org/10.1007/978-1-4757-1080-9_3}{doi:10.1007/978-1-4757-1080-9\_3}.

\bibitem[Bre23]{Brezis2023}
H.~Brezis,
\emph{Some of my favorite open problems},
Rend. Lincei Mat. Appl. \textbf{34} (2023), no.~2, 307--335.
\href{https://doi.org/10.4171/RLM/1008}{doi:10.4171/RLM/1008}.

\bibitem[BC83]{BrezisCoron1983}
H.~Brezis and J.-M.~Coron,
\emph{Large solutions for harmonic maps in two dimensions},
Comm. Math. Phys. \textbf{92} (1983), no.~2, 203--215.
\href{https://doi.org/10.1007/BF01210846}{doi:10.1007/BF01210846}.

\bibitem[BL01]{BrezisLi2001}
H.~Brezis and Y.~Li,
\emph{Topology and Sobolev spaces},
J. Funct. Anal. \textbf{183} (2001), no.~2, 321--369.
\href{https://doi.org/10.1006/jfan.2000.3736}{doi:10.1006/jfan.2000.3736}.

\bibitem[BMS16]{BrezisMironescuShafrir2016}
H.~Brezis, P.~Mironescu, and I.~Shafrir,
\emph{Distances between homotopy classes of $W^{s,p}(\mathbb S^N;\mathbb S^N)$},
ESAIM Control Optim. Calc. Var. \textbf{22} (2016), no.~4, 1204--1235.
\href{https://doi.org/10.1051/cocv/2016037}{doi:10.1051/cocv/2016037}.

\bibitem[BMS18]{BrezisMironescuShafrir2018}
H.~Brezis, P.~Mironescu, and I.~Shafrir,
\emph{Distances between classes in $W^{1,1}(\Omega;\Sph^1)$},
Calc. Var. Partial Diﬀerential Equations \textbf{57} (2018) no.~1, Art. 14, 32. \href{https://doi.org/10.1007/s00526-017-1280-z}{doi:10.1007/s00526-017-1280-z}

\bibitem[BN95]{BrezisNirenberg1995}
H.~Brezis and L.~Nirenberg,
\emph{Degree theory and BMO. I. Compact manifolds without boundaries},
Selecta Math. (N.S.) \textbf{1} (1995), no.~2, 197--263.
\href{https://doi.org/10.1007/BF01671566}{doi:10.1007/BF01671566}.

\bibitem[DVS26]{DetailleVanSchaftingen}
A.~Detaille and J.~Van Schaftingen,
\emph{Heterotopic energy for Sobolev mappings},
Commun. Contemp. Math. \textbf{28} (2026), no.~5, Paper No.~2640006.
\href{https://doi.org/10.1142/S0219199726400067}{doi:10.1142/S0219199726400067}.

\bibitem[LS14]{LeviShafrir2014}
S.~Levi and I.~Shafrir,
\emph{On the distance between homotopy classes of maps between spheres},
J. Fixed Point Theory Appl. \textbf{15} (2014), no.~2, 501--518.
\href{https://doi.org/10.1007/s11784-014-0156-5}{doi:10.1007/s11784-014-0156-5}.

\bibitem[RS07]{RubinsteinShafrir2007}
J.~Rubinstein and I.~Shafrir,
\emph{The distance between homotopy classes of $\mathbb S^1$-valued maps in multiply connected domains},
Israel J. Math. \textbf{160} (2007), 41--59.
\href{https://doi.org/10.1007/s11856-007-0055-1}{doi:10.1007/s11856-007-0055-1}.

\bibitem[SU83]{SchoenUhlenbeck1983}
R. Schoen and K. Uhlenbeck, 
\emph{Boundary regularity and the Dirichlet problem for harmonic maps}, J. Diﬀer. Geom. \textbf{18} (1983), 253--268.
\href{https://doi.org/10.4310/jdg/1214437663}{doi:10.4310/jdg/1214437663}

\bibitem[Sha18]{Shafrir2018}
I.~Shafrir,
\emph{On the distance between homotopy classes in $W^{1/p,p}(\mathbb S^1;\mathbb S^1)$},
Confluentes Math. \textbf{10} (2018), no.~1, 125--136.
\href{https://doi.org/10.5802/cml.48}{doi:10.5802/cml.48}.

\end{thebibliography}
\end{document}